\documentclass[11pt,amsfonts]{article}
\usepackage[margin=1in]{geometry}  
\usepackage{graphicx}              

\usepackage{amsmath, amssymb, amsthm, epsfig}               
\usepackage{enumerate}
\usepackage{amsfonts}              
\usepackage{amsthm}                
\usepackage{amsthm}
\usepackage{enumerate}
\theoremstyle{plain}
\usepackage{color}
\newtheorem{corollary}{Corollary}[section]
\newtheorem{definition}[corollary]{Definition}

\newtheorem{lemma}[corollary]{Lemma}
\newtheorem{prp}[corollary]{Proposition}
\newtheorem{remark}[corollary]{Remark}
\newtheorem{thm}[corollary]{Theorem}

\newfont{\sBlackboard}{msbm10 scaled 900}

\newcommand{\mylabel}[1]{\label{#1}
	\ifx\undefined\stillediting
	\else \fbox{$#1$}\fi }
\newcommand{\BE}{\begin{equation}}

\newcommand{\EEQ}{\end{equation}}
\newcommand{\rfb}[1]{\mbox{\rm
		(\ref{#1})}\ifx\undefined\stillediting\else:\fbox{$#1$}\fi}

\newfont{\Blackboard}{msbm10 scaled 1200}

\newfont{\roma}{cmr10 scaled 1200}

\newcommand{\bb}{\begin{equation}}
\newcommand{\bbb}{\end{equation}}

\newcommand{\mm}    {{\hbox{\hskip 0.5pt}}}

\newcommand{\bluff} {{\hbox{\raise 15pt \hbox{\mm}}}}
scaled\magstep2

\makeatletter
\def\section{\@startsection {section}{1}{\z@}{-3.5ex plus -1ex minus
		-.2ex}{2.3ex plus .2ex}{\large\bf}}
\makeatother

\begin{document}
\title{Basic results of fractional Orlicz-Sobolev space and applications to non-local problems}
\author{Sabri Bahrouni, Hichem Ounaies and Leandro S. Tavares  }


\maketitle

\begin{abstract}
In this paper, we study the interplay between Orlicz-Sobolev spaces $L^{M}$ and $W^{1,M}$ and fractional Sobolev spaces $W^{s,p}$. More precisely, we give some qualitative properties of the new fractional Orlicz-Sobolev space $W^{s,M}$, where $s\in (0,1)$ and $M$ is a Young function. We also study a related non-local operator, which is a fractional version of the nonhomogeneous $M$-Laplace operator. As an application, we prove existence of weak solution for a non-local problem involving the new fractional $M-$Laplacian operator.
\end{abstract}

{\small \textbf{2010 Mathematics Subject Classification:} Primary: 35J60; Secondary: 35J91, 35S30, 46E35, 58E30.}

{\small \textbf{Keywords:} Fractional Orlicz-Sobolev space, Fractional $M-$Laplacian, non-local problems, existence of solution}


\section{Introduction}
Recently, great attention has been focused on the study of fractional and non-local operators of elliptic type, both for pure mathematical research and in view of concrete real-world applications. In most of these applications a fundamental tool to treat these type of problems
is the so-called fractional order Sobolev spaces that for $0<s<1\leq p<\infty,$ are defined as
$$W^{s,p}(\Omega)=\bigg{\{}u\in L^{p}(\Omega):\ \frac{u(x)-u(y)}{|x-y|^{\frac{N}{p}+s}}\in  L^{p}(\Omega\times\Omega)\bigg{\}},$$
where $\Omega\subset\mathbb{R}^{N}$  is an open set.
The literature on non-local operators and on their applications is very interesting and, up to now, quite large. After the seminal papers by Caffarelli {\it et al.} \cite{11,12,13}, a
large amount of papers were written on problems involving the
fractional diffusion operator $(-\Delta)^{s}$  ($0<s<1$). We can quote \cite{4,14,21,23,25,26} and the references therein.  We also refer to the recent monographs \cite{14,22} for a thorough variational approach of  non-local problems.\\
On the other hand, for some nonhomogeneous materials
(such as electrorheological fluids, sometimes referred to as ``smart
fluids"), the standard approach based on Lebesgue and Sobolev spaces $L^p$ and $W^{1,p}$, is not adequate.
This leads to the study of {\it variable exponent} Lebesgue and Sobolev spaces, $L^{p(x)}$ and $W^{1,p(x)}$, where $p$ is a real-valued function.
Variable exponent Lebesgue spaces appeared in the literature in $1931$ in the paper by Orlicz \cite{orlicz}. We refer the reader to \cite{Die, Fan, radrep,Ruz, zi} for more details on Sobolev space with variable exponent.\\

A natural question is to see what results can be recovered when the standard
$p(x)$-Laplace operator is replaced by the fractional $p(x)-$Laplacian. It is worth mentioning that there are some papers concerning related equations involving the fractional $p(x)-$Laplace operator. In fact, results for the fractional Sobolev spaces with variable exponent and fractional
$p(x)-$Laplace equations are few, for example, we refer to \cite{5,6,16,27}.

In the theory of PDEs, when trying to relax some conditions on the operators, such as growth conditions, the problem can not be formulated with classical Lebesgue and Sobolev spaces with variable exponents. Hence, the adequate functional spaces is the so-called Orlicz spaces. More precisely, if in the definition of the ordinary Sobolev space $W^{1,p}(\Omega)$, the role played by the Lebesgue space $L^{p}(\Omega)$ is assumed instead by a more general Orlicz space $L^{M}(\Omega)$, the resulting structure is called an Orlicz-Sobolev space and denoted $W^{1,M}(\Omega)$, where $M$ is a Young function admitting an integral representation $M(t)=\displaystyle\int_{0}^{|t|}m(s)ds$ and $m$ assumes some conditions (see section $2$). Classical Sobolev and Orlicz-Sobolev spaces play a significant role in many fields of mathematics, such as partial differential equations. For more details on the theory of Orlicz and Orlicz-Sobolev, we can cite \cite{1,9, Fuk, Garcia,15,24} and the references therein.\\

It is therefore a natural question to see what results can be ``recovered"
when the $M$-Laplace operator is replaced by the fractional
$M$-Laplacian. As far as we know, the only results about the
fractional Orlicz-Sobolev spaces and the fractional
$M$-Laplacian are obtained in \cite{7,Salort}. In particular,  the
authors generalize the $M$-Laplace operator to the fractional case.
They also introduce a suitable functional space to study an equation
in which a fractional $M-$Laplace operator is present.\\

A bridge between fractional order theories and Orlicz-Sobolev settings is provided in \cite{7}, where the authors define the fractional order Orlicz-Sobolev space associated to a Young function $M$ and a fractional parameter $0<s<1$ as
$$W^{s,M}(\Omega)=\bigg{\{}u\in L^{M}(\Omega):\ \int_{\Omega}\int_{\Omega}M\bigg{(}\frac{u(x)-u(y)}{|x-y|^{s}}\bigg{)}\frac{dxdy}{|x-y|^{N}}<\infty\bigg{\}}.$$

 They define the fractional $M$-Laplacian operator as,
\begin{equation}\label{newoperator}
(-\triangle)^{s}_{m}u(x)=P.V.\int_{\mathbb{R}^{N}} m\bigg{(}\frac{u(x)-u(y)}{|x-y|^{s}}\bigg{)}\frac{dy}{|x-y|^{N+s}},
\end{equation}
this operator is a direct generalization of the fractional $p$-Laplacian. They also deduce some consequences such as $\Gamma$-convergence of the
modulars and convergence of solutions for some fractional versions of the $(\triangle)^{s}_{m}$ operator as the fractional parameter $s\uparrow1$.\\

Below we point out several operators that can be incorporated to \eqref{newoperator} by using the following functions  which satisfy the hypotheses that will be considered in this work,
\begin{itemize}
	\item	 $M (t)=|t|^p$ for $2\leq p < N$
	\item $M(t)=|t|^p+ |t|^q$ for  $2 \leq p <  q < N$ and $q \in ]p,p^{\star}[$ with $p^{\star}:= \frac{Np}{N-p}$
	\item $M (t)= (1+|t|^{2})^{\gamma } -1,$ for  $1 <  \gamma < \frac{N}{N-2}.$
\end{itemize}

The  Young functions $M$ associated to the above functions arise in several areas, for example quantum-physics and nonlinear elasticity problems, see for instance \cite{Benci-Fortunato,FN}.

The main purpose of this paper is to
present some further basic results both on the function spaces $W^{s,M}(\Omega)$ and the
fractional $M-$Laplace operator. Then, we study the existence of solutions to the non-local problem
\begin{equation}\label{eq}
\begin{cases}
(-\triangle)^{s}_{m}u=\lambda g(x,u) & \mbox{in }\ \ \Omega, \\
u=0 & \mbox{in}\ \ \mathbb{R}^{N}\setminus\Omega,
\end{cases}
\end{equation}
where $\Omega$ is a bounded domain in $\mathbb{R}^{N}$ with smooth boundary $\partial\Omega$, $0<s<1,$ $\lambda>0$ is a parameter and the nonlinear term $g:\Omega\times\mathbb{R}\rightarrow\mathbb{R}$ is a Carath\'eodory function that satisfy\\

\noindent$(A)$ $\quad\ \ \ |g(x,t)|\leq C_{0} |t|^{q-1},\ \forall x\in\Omega,\ t\in\mathbb{R}.$\\ \\
$(B)$ $\ \ \ \quad \displaystyle C_{1}|t|^{q}\leq G(x,t):=\int_{0}^{t}g(x,s)ds\leq \displaystyle C_{2}|t|^{q},\ \forall x\in\Omega,\ t\in\mathbb{R} ,$\\

where $C_{0}$, $C_{1}$ and  $C_{2}$ are positive constants and $1<q<p^{*}=\displaystyle\frac{Np}{N-p}$.\\ \\
$(Q)$ \ \ $\lim_{t\rightarrow+\infty}\displaystyle\frac{|t|^{q}}{M(t)}$=0.\\
Regarding the hypotheses $(A)$ and $(B)$  we point out that the following functions $g$ and $G$ satisfy such hypotheses:
	\begin{enumerate}
	\item 	$g(x,t)=q|t|^{q-2}t$ and $G(x,t)=|t|^{q}$, where $2<q<p^{*}$ for all $x\in\Omega$.\\
		
		\item $g(x,t)=q|t|^{q-2}t+(q-2)[\log(1+t^{2})]|t|^{q-4}t+\frac{1}{1+t^{2}}|t|^{q-2}$ and $G(x,t)=\log(1+t^{2})|t|^{q-2}$, where $4<q<p^{*}$ for all $x\in\Omega$.
	\end{enumerate}

Throughout this paper we assume that $M$ is a Young function satisfying

\begin{equation}\label{10}
1<\displaystyle \inf_{t>0}\frac{tm(t)}{M(t)}\leq \displaystyle\sup_{t>0}\frac{tm(t)}{M(t)}<\infty. 
\end{equation}
Due to assumption \eqref{10}, we may define the numbers $$m_{0}=\displaystyle \inf_{t>0}\frac{tm(t)}{M(t)}\ \ \text{and}\ \ m^{0}=\displaystyle\sup_{t>0}\frac{tm(t)}{M(t)}.$$
We also assume that the function $M$ satisfies the following condition:\\

\noindent$(S)$ $\quad\ \ \ \text{the function}\ \ t\mapsto M(\sqrt{t}),\ t\in[0,\infty[\ \ \text{ is convex}.$\\

Our main result is given by the following theorem.
\begin{thm}\label{thm}
	Suppose that  $(A)$, $(B)$, $(Q)$, $(S)$ and \eqref{10} are satisfied. Furthermore, we assume that $q<\min(p^{*}, m_{0})$. Then there exists $\lambda^{*}>0$ such that for any $\lambda\in ]0,\lambda^{*}[$ problem \eqref{eq} has at least two distinct, non-trivial weak solutions.
\end{thm}

This paper is organized as follows. In Section $2$, we give some
definitions and fundamental properties of the spaces
$L^{M}(\Omega)$ and $W^{1,M}(\Omega)$. In Section $3$,  we prove some basic properties of the fractional Orlicz-Sobolev space and it's associated operator. Finally, in Section $4$, using a direct variational method, we give an application of our abstract results.

\section{Preliminaries}

In this preliminary section, for the reader's convenience, we make a brief overview on the classical Orlicz-Sobolev spaces, as well as we introduce the Fractional Orlicz-Sobolev Spaces, studied in \cite{7}, and the associated fractional $M$-laplacian operator.

\subsection{Orlicz and Orlicz-Sobolev spaces}

We start by recalling some basic facts about Orlicz spaces.

Let $\Omega$ be an open subset of $\mathbb{R}^{N}.$ Let $M: \mathbb{R}\rightarrow\mathbb{R_{+}}$ be a Young function, i.e,

\begin{enumerate}
	\item $M$ is even, continuous, convex, with $M(t)>t$ for $t>0$,
	\item $\frac{M(t)}{t}\rightarrow 0$ as $t\rightarrow0$ and  $\frac{M(t)}{t}\rightarrow +\infty$ as $t\rightarrow+\infty$.
\end{enumerate}

Equivalently, $M$ admits the representation: $$M(t)=\int_{0}^{|t|}m(s)ds,$$ where $m: \mathbb{R}\rightarrow\mathbb{R}$ is non-decreasing, right continuous, with $m(0)=0$, $m(t)>0\ \forall t>0$ and $m(t)\rightarrow\infty$ as $t\rightarrow\infty$. The conjugate Young function of $M$ is defined by
$$\overline{M}(t)=\int_{0}^{|t|}\overline{m}(s)ds,$$ where  $\overline{m}: \mathbb{R}\rightarrow\mathbb{R}$ is given by
$\overline{m}(t)=\sup\{s:\ m(s)\leq t\}$.
Evidently we have

\begin{equation}\label{9}
st\leq M(s)+\overline{M}(t),
\end{equation}

which is known as the Young inequality. Equality holds in \eqref{9} if and only if either $t=m(s)$ or $s=\overline{m}(t)$.\\

If $A$ and $B$ are two Young functions, we say that $A$ is essentially stronger than $B$ if $$B(x)\leq A(ax),\ x\geq x_{0}\geq 0,$$ for each $a>0$ and $x_{0}$ (depending on $a$), $B\prec\prec A$ in symbols. This is the case if and only if for every positive constante $k$ $$\lim_{t\rightarrow+\infty}\frac{B(kt)}{A(t)}=0.$$

Under the condition \eqref{10} we have that $M$ and $\overline{M}$ satisfy the $\triangle_{2}$-condition, i.e.

\begin{equation}\label{16}
M(2t)\leq K M(t)\ \ \forall\ t\geq0.
\end{equation}

Considering that
 $$m\ \text{is an increasing homeomorphism from}\ \mathbb{R}\ \text{onto}\ \mathbb{R},$$
we have that $\overline{M}$ became $$\overline{M}(t)=\int_{0}^{|t|}m^{-1}(s)ds.$$
The Orlicz class $K^{M}(\Omega)$ (resp. the Orlicz space $L^{M}(\Omega)$) is defined as the set of (equivalence classes of) real-valued measurable
functions $u$ on $\Omega$ such that $$\rho(u;M)=\int_{\Omega}M(u(x))dx<\infty\ (\text{resp.}\ \int_{\Omega}M(\lambda u(x))dx<\infty\ \text{for some}\ \lambda>0).$$
$L^{M}(\Omega)$ is a Banach space under the Luxemburg norm
\begin{equation}\label{19}
\|u\|_{(M)}=\inf\bigg{\{}\lambda>0\ :\ \int_{\Omega}M(\frac{u}{\lambda})\leq1\bigg{\}},
\end{equation}
and $K^{M}(\Omega)$ is a convex subset of  $L^{M}(\Omega)$.

\begin{prp}\label{cv}
	Let $(u_{n})_{n\in\mathbb{N}}$ be a sequence in $L^{M}$ and $u\in L^{M}$.
	If $M$ satisfies the $\triangle_{2}$-condition and $\rho(u_{n};M)\rightarrow\rho(u;M)$, then $u_{n}\rightarrow u$ in $ L^{M}$ .
\end{prp}

Next, we introduce the Orlicz-Sobolev spaces. We denote by $W^{1,M}(\Omega)$ the Orlicz-Sobolev space defined by
$$W^{1,M}(\Omega):=\bigg{\{}u\in L^{M}(\Omega):\ \frac{\partial u}{\partial x_{i}}\in L^{M}(\Omega),\ i=1,...,N\bigg{\}}.$$
This is a Banach space with respect to the norm
$$\|u\|_{1,M}=\|u\|_{(M)}+ \||\nabla u|\|_{(M)}.$$

\subsection{Fractional Orlicz-Sobolev spaces}

\begin{definition}
	Let $M$ be a Young function. For a given domain $\Omega$ in $\mathbb{R}^{N}$ and $0<s<1$, we define the fractional Orlicz-Sobolev space $W^{s,M}(\Omega)$ as follows,
	\begin{equation}\label{20}
	W^{s,M}(\Omega)=\bigg{\{}u\in
L^{M}(\Omega):\ \int_{\Omega}\int_{\Omega}M\bigg{(}\frac{u(x)-u(y)}{|x-y|^{s}}\bigg{)}\frac{dxdy}{|x-y|^{N}}<\infty\bigg{\}}.
	\end{equation}
	This space is equipped with the norm,
	\begin{equation}\label{21}
	\|u\|_{(s,M)}=\|u\|_{(M)}+[u]_{(s,M)},
	\end{equation}
	where $[.]_{(s,M)}$ is the Gagliardo semi-norm, defined by
	\begin{equation}\label{22}
	[u]_{(s,M)}=\inf\bigg{\{}\lambda>0:\ \int_{\Omega}\int_{\Omega} M\bigg{(}\frac{u(x)-u(y)}{\lambda|x-y|^{s}}\bigg{)}\frac{dxdy}{|x-y|^{N}}\leq1\bigg{\}}.
	\end{equation}
\end{definition}

\begin{prp}[\cite{7}]
  Let $M$ be a Young function such that $M$ and $\overline{M}$ satisfy the $\triangle_{2}$-condition, and consider $s\in(0,1)$. Then $W^{s,M}(\mathbb{R}^{N})$ is a reflexive and separable Banach space. Moreover, $C^{\infty}_{0}(\mathbb{R}^{N})$ is dense in $W^{s,M}(\mathbb{R}^{N})$.
\end{prp}

A variant of the well-known Fr\`{e}chet-Kolmogorov compactness theorem gives the compactness of the
inclusion of $W^{s,M}$ into $L^{M}$.

\begin{thm}[\cite{7}]\label{thm1}
  Let $M$ be a Young function, $s\in(0,1)$ and $\Omega$ a bounded open set in $\mathbb{R}^{N}$. Then the embedding $$W^{s,M}(\Omega)\hookrightarrow L^{M}(\Omega)$$ is compact.
\end{thm}

Let $W^{s,M}_{0}(\Omega)$ denote the closure of $C_{c}^{\infty}(\Omega)$ in the norm $\|.\|_{(s,M)}$ defined in \eqref{21}.

\begin{thm}\label{thm2}\cite{Salort}(Generalized Poincar\'{e} inequality)
Let $\Omega$ be a bounded open subset of $\mathbb{R}^{N}$ and let $s\in]0,1[$. Let $M$ be a Young function. Then there exists a positive
	constant $\mu$ such that, $$\|u\|_{(M)}\leq \mu [u]_{(s,M)},\ \ \forall\ u\in W^{s,M}_{0}(\Omega).$$
\end{thm}
Therefore, if $\Omega$ is bounded and $M$ be a \textcolor{blue}{Young} function, then $[u]_{(s,M)}$ is a norm of $ W^{s,M}_{0}(\Omega)$ equivalent to $\|u\|_{(s,M)}$.\\

The fractional $M$-Laplacian operator is defined as
\begin{equation}\label{17}
(-\triangle)^{s}_{m}u(x)= P.V.\int_{\mathbb{R}^{N}} m\bigg{(}\frac{u(x)-u(y)}{|x-y|^{s}}\bigg{)}\frac{dy}{|x-y|^{N+s}},
\end{equation}
where $P.V.$ is the principal value.\\
This operator is well defined between $W^{s,M}(\mathbb{R}^{N})$ and its dual space  $W^{-s,\overline{M}}(\mathbb{R}^{N})$. In fact, in [\cite{7}, Theorem 6.12] the following representation formula is provided
\begin{equation}\label{18}
\langle(-\triangle)^{s}_{m}u,v\rangle=\int_{\mathbb{R}^{N}}\int_{\mathbb{R}^{N}} m\bigg{(}\frac{u(x)-u(y)}{|x-y|^{s}}\bigg{)}\frac{v(x)-v(y)}{|x-y|^{s}}\frac{dxdy}{|x-y|^{N}},
\end{equation}
for all $v\in W^{s,M}(\mathbb{R}^{N})$.

\section{Basic results of $W^{s,M}(\mathbb{R}^{N})$ and fractional $M-$Laplacian operator}

In this section, we point out certain useful auxiliary results.\\ Let $E$ denote the generalized Sobolev space $W^{s,M}(\mathbb{R}^{N})$. We define the functional $F: E\rightarrow \mathbb{R}$ by
\begin{equation}\label{F}
F(u)=\int_{\mathbb{R}^{N}}\int_{\mathbb{R}^{N}}M\bigg{(}\frac{u(x)-u(y)}{|x-y|^{s}}\bigg{)}\frac{dxdy}{|x-y|^{N}}.
\end{equation}

\begin{lemma}\label{lem2}
	The following properties hold true:\\ \\
$(i)$ $F\bigg{(}\displaystyle\frac{u}{[u]_{(s,M)}}\bigg{)}\leq1, \ \ \ \forall\ u\in E\backslash\{0\};$\\ \\
$(ii)$ $[u]_{(s,M)}^{m_{0}}\leq F(u) \leq [u]_{(s,M)}^{m^{0}}\ \ \forall\ u\in E,\ \ [u]_{(s,M)}>1;$\\ \\
$(iii)$ $[u]_{(s,M)}^{m^{0}}\leq F(u) \leq [u]_{(s,M)}^{m_{0}}\ \ \forall\ u\in E,\ \ [u]_{(s,M)}<1.$\\ \\
\end{lemma}

\begin{proof}

$(i)$ Let $(\lambda_{k})$ be a sequence such that $\lambda_{k} \rightarrow [u]_{(s,M)}$ as $k\rightarrow \infty$. Then, the definition of the norm, yields
	$$\int_{\mathbb{R}^{N}}\int_{\mathbb{R}^{N}} M\bigg{(}\frac{u(x)-u(y)}{\lambda_{k}|x-y|^{s}}\bigg{)}\frac{dxdy}{|x-y|^{N}}\leq1.$$
	Passing by limit in the above inequality and using Fatou's Lemma, we can deduce that  $$\int_{\mathbb{R}^{N}}\int_{\mathbb{R}^{N}} M\bigg{(}\frac{1}{[u]_{(s,M)}}\frac{u(x)-u(y)}{|x-y|^{s}}\bigg{)}\frac{dxdy}{|x-y|^{N}}\leq1.$$
	$(ii)$ Since $\displaystyle\frac{tm(t)}{M(t)}\leq m^{0}$ for all $t>0$, it follows that for all $\sigma>1$,
$$\log(M(\sigma t))-\log(M(t))=\int_{t}^{\sigma t}\frac{m(\tau)}{M(\tau)}d\tau\leq \int_{t}^{\sigma t}\frac{m^{0}}{\tau}d\tau=\log(\sigma^{m^{0}}).$$
Thus we deduce
\begin{equation}\label{Msigma1}
  M(\sigma t)\leq\sigma^{m^{0}}M(t)\ \ \text{for all}\ t>0\ \text{and}\ \sigma>1.
\end{equation}
Let $u\in E$ and $[u]_{(s,M)}>1$. Using the definition of Luxemburg norm and the relation \eqref{Msigma1}, we deduce
\begin{align*}
  \int_{\mathbb{R}^{N}}\int_{\mathbb{R}^{N}}M\bigg{(}\frac{u(x)-u(y)}{|x-y|^{s}}\bigg{)}\frac{dxdy}{|x-y|^{N}} & =
  \int_{\mathbb{R}^{N}}\int_{\mathbb{R}^{N}}M\bigg{(}[u]_{(s,M)}\frac{u(x)-u(y)}{[u]_{(s,M)}|x-y|^{s}}\bigg{)}\frac{dxdy}{|x-y|^{N}}\\
  &\leq [u]_{(s,M)}^{m^{0}}\int_{\mathbb{R}^{N}}\int_{\mathbb{R}^{N}}M\bigg{(}\frac{u(x)-u(y)}{[u]_{(s,M)}|x-y|^{s}}\bigg{)}\frac{dxdy}{|x-y|^{N}}\\
  &\leq [u]_{(s,M)}^{m^{0}}.
\end{align*}
Now, since $m_{0}\leq\frac{tm(t)}{M(t)}$ for all $t>0$, it follows that for all $\sigma>1$,
$$\log(M(\sigma t))-\log(M(t))=\int_{t}^{\sigma t}\frac{m(\tau)}{M(\tau)}d\tau\geq \int_{t}^{\sigma t}\frac{m_{0}}{\tau}d\tau=\log(\sigma^{m_{0}}).$$
Hence, we deduce
\begin{equation}\label{Msigma2}
  M(\sigma t)\geq\sigma^{m_{0}}M(t)\ \ \text{for all}\ t>0\ \text{and}\ \sigma>1.
\end{equation}
Let $u\in E$ and $[u]_{(s,M)}>1$, we consider $1<\beta<[u]_{(s,M)}$ so by definition of Luxemburg norm, it follows that
$$ \int_{\mathbb{R}^{N}}\int_{\mathbb{R}^{N}}M\bigg{(}\frac{u(x)-u(y)}{\beta|x-y|^{s}}\bigg{)}\frac{dxdy}{|x-y|^{N}}>1,$$
the above inequality implies that
\begin{align*}
  \int_{\mathbb{R}^{N}}\int_{\mathbb{R}^{N}}M\bigg{(}\frac{u(x)-u(y)}{|x-y|^{s}}\bigg{)}\frac{dxdy}{|x-y|^{N}} & =
  \int_{\mathbb{R}^{N}}\int_{\mathbb{R}^{N}}M\bigg{(}\beta\frac{u(x)-u(y)}{\beta|x-y|^{s}}\bigg{)}\frac{dxdy}{|x-y|^{N}}\\
  &\geq \beta^{m_{0}}\int_{\mathbb{R}^{N}}\int_{\mathbb{R}^{N}}M\bigg{(}\frac{u(x)-u(y)}{\beta|x-y|^{s}}\bigg{)}\frac{dxdy}{|x-y|^{N}}\\
  &\geq \beta^{m_{0}}.
\end{align*}
The estimate in $(ii)$ follows letting $\beta\nearrow[u]_{(s,M)}$.

$(iii)$ By the same argument in the proof of \eqref{Msigma1} and \eqref{Msigma2}, we have
\begin{equation}\label{Msigma3}
  M(t)\leq \tau^{m_{0}}M\bigg{(}\displaystyle\frac{t}{\tau}\bigg{)}\ \ \text{for all}\ t>0,\ \tau\in(0,1).
\end{equation}
Let $u\in E$ and $[u]_{(s,M)}<1$. Using the definition of Luxemburg-norm and the relation \eqref{Msigma3}, we deduce
\begin{align*}
   \int_{\mathbb{R}^{N}}\int_{\mathbb{R}^{N}}M\bigg{(}\frac{u(x)-u(y)}{|x-y|^{s}}\bigg{)}\frac{dxdy}{|x-y|^{N}} & =
   \int_{\mathbb{R}^{N}}\int_{\mathbb{R}^{N}}M\bigg{(}[u]_{(s,M)}\frac{u(x)-u(y)}{[u]_{(s,M)}|x-y|^{s}}\bigg{)}\frac{dxdy}{|x-y|^{N}} \\
   &\leq [u]_{(s,M)}^{m_{0}}\int_{\mathbb{R}^{N}}\int_{\mathbb{R}^{N}}M\bigg{(}\frac{u(x)-u(y)}{[u]_{(s,M)}|x-y|^{s}}\bigg{)}\frac{dxdy}{|x-y|^{N}}\\
   &\leq [u]_{(s,M)}^{m_{0}}.
\end{align*}
Similar arguments in the proof of \eqref{Msigma1} and \eqref{Msigma2}, we have
\begin{equation}\label{Msigma4}
  M(t)\geq \tau^{m^{0}}M\bigg{(}\displaystyle\frac{t}{\tau}\bigg{)}\ \ \text{for all}\ t>0,\ \tau\in(0,1).
\end{equation}
Let $u\in E$ with $[u]_{(s,M)}<1$ and $\beta<[u]_{(s,M)}<1$, so by \eqref{Msigma4} we have
\begin{align*}
  \int_{\mathbb{R}^{N}}\int_{\mathbb{R}^{N}}M\bigg{(}\frac{u(x)-u(y)}{|x-y|^{s}}\bigg{)}\frac{dxdy}{|x-y|^{N}} & \geq
  \beta^{m^{0}}\int_{\mathbb{R}^{N}}\int_{\mathbb{R}^{N}}M\bigg{(}\frac{u(x)-u(y)}{\beta^{m^{0}}|x-y|^{s}}\bigg{)}\frac{dxdy}{|x-y|^{N}}\\
  &\geq \beta^{m^{0}}.
\end{align*}
The estimate in $(iii)$ follows letting $\beta\nearrow[u]_{(s,M)}$. This ends the proof.
\end{proof}

\begin{lemma}\label{lem1}
	The functional $F$ is of class $C^{1}(E,\mathbb{R})$ and
	\begin{align*}
	\langle F^{'}(u),v\rangle&=\int_{\mathbb{R}^{N}}\int_{\mathbb{R}^{N}} m\bigg{(}\frac{u(x)-u(y)}{|x-y|^{s}}\bigg{)}\frac{v(x)-v(y)}{|x-y|^{s}}\frac{dxdy}{|x-y|^{N}}\\
	& =\langle(-\triangle)^{s}_{m}u,v\rangle.
	\end{align*}
\end{lemma}
\begin{proof}
  The proof is given by Proposition $3.3$ in \cite{Salort}.
\end{proof}

\begin{lemma}\label{lem4}
	The functional $F$ is weakly lower semi-continuous.
\end{lemma}

\begin{proof}
First observe that if we denote $d\mu=\frac{dxdy}{|x-y|^{N}}$, then $d\mu$ is a regular Borel measure on the set $\Omega\times\Omega$ and the spaces $L^{M}(d\mu)$ and $L^{\overline{M}}(d\mu)$ are reflexive and separable Banach spaces when endowed with the norms
$$\| w \|^{(M)}:= \inf \left\{\lambda >0 ; \int_{\Omega } \int_{\Omega}M \left( \frac{w(x,y)}{\lambda}\right) \frac{dxdy}{|x-y|^{N}} \leq 1 \right\} $$
and
$$\| z \|^{(M)}:= \inf \left\{\lambda >0 ; \int_{\Omega } \int_{\Omega}M \left( \frac{z(x,y)}{\lambda}\right) \frac{dxdy}{|x-y|^{N}} \leq 1\right\} $$
respectively.

By Corollary $III.8$ in \cite{8}, it is enough to show that $F$ is inferior semi-continuous. For this
	purpose, we fix $u\in E$ and $\epsilon>0$. Since $F$ is convex, we deduce that for any $v\in E$ the following inequality holds
	$$F(v)\geq F(u)+\langle F^{'}(u),v-u\rangle.$$
	Using H\"{o}lder inequality we have
	\begin{align*}
	F(v)&\geq F(u)-\langle F^{'}(u),u-v\rangle\\
	&= F(u)- \int_{\Omega}\int_{\Omega}m(h_{u})h_{u-v}\frac{dxdy}{|x-y|^{N}}\\
	&= F(u)- \int_{\Omega}\int_{\Omega}m(h_{u})h_{u-v}d\mu\\
	&\geq F(u)-\|m(h_{u})\|^{(\overline{M})}\|h_{u-v}\|^{(M)}\\
    &=F(u)-\|m(h_{u})\|^{(\overline{M})}[u-v]_{(M)}\\
	&\geq F(u)-\|m(h_{u})\|^{(\overline{M})}\|u-v\|_{(s,M)}\\
	&= F(u)-C\|u-v\|_{(s,M)}
	\geq F(u)-\epsilon
	\end{align*}
	for all $v\in E$ with $\|u-v\|_{(s,M)}<\delta=\displaystyle\frac{\epsilon}{C}$, where $C$ is positive constant and $h_{u}:=\frac{u(x)-u(y)}{|x-y|^{s}}$. We conclude that $F$ is weakly lower semi-continuous.
\end{proof}

\begin{lemma}\label{lem3}
	Suppose that $(S)$ is fulfilled. Moreover, we assume that the sequence $(u_{n})$ converges weakly to $u$ in $E$ and
	\begin{equation}\label{1}
	\displaystyle\limsup_{n\rightarrow+\infty}\langle F^{'}(u_{n}),\ u_{n}-u\rangle\leq 0.
	\end{equation}
	Then $(u_{n})$ converges strongly to $u$ in $E$.
\end{lemma}

\begin{proof}
	Since $(u_{n})$ converges weakly to $u$ in $E$ implies that $([u_{n}]_{(s,M)})$ is a bounded sequence of real numbers. That fact and relations $(ii)$ and $(iii)$ from lemma \ref{lem2} imply that the sequence $(F(u_{n}))$ is bounded. Then, up to a subsequence, we deduce that $F(u_{n})\rightarrow c$.
	Furthermore, the weak lower semi-continuity of $F$ implies
	\begin{equation}\label{23}
	F(u)\leq \displaystyle\liminf_{n\rightarrow\infty}
	F(u_{n})=c.
	\end{equation}
	On the other hand, since $F$ is convex, we have \begin{equation}\label{24}F(u)\geq F(u_{n})+\langle F^{'}(u_{n}),u-u_{n}\rangle.\end{equation}
	Therefore, combinings \eqref{23} and \eqref{24} and the hypothesis \eqref{1}, we conclude that $F(u)=c$.\\
	Taking into account that $\displaystyle\frac{u_{n}+u}{2}$ converges weakly to $u$ in $E$ and using again the weak lower semi-continuity of $F$ we find
	\begin{equation}\label{11}
	c=F(u)\leq \displaystyle\liminf_{n\rightarrow\infty}F\bigg{(}\frac{u_{n}+u}{2}\bigg{)}.
	\end{equation}
	We assume by contradiction that $(u_{n})$ does not converge to $u$ in $E$. Then by $(i)$ in lemma \ref{lem2} it follows that there exist $\epsilon>0$ and a subsequence $(u_{n_{m}})$ of $(u_{n})$ such that
	\begin{equation}\label{12}
	F\bigg{(}\frac{u_{n_{m}}-u}{2}\bigg{)}\geq\epsilon\ \forall\ m\in\mathbb{N}.
	\end{equation}
	On the other hand, relations \eqref{16} and $(S)$ enable us to apply [\cite{17}, theorem 2.1] in order to obtain
	\begin{equation}\label{13}
	\frac{1}{2}F(u)+\frac{1}{2}F(u_{n_{m}})-F\bigg{(}\frac{u_{n_{m}}+u}{2}\bigg{)}\geq F\bigg{(}\frac{u_{n_{m}}-u}{2}\bigg{)}\geq\epsilon,\ \forall m\in\mathbb{N}.
	\end{equation}
	Letting $m\rightarrow\infty$ in the above inequality we obtain
	\begin{equation}\label{14}
	c-\epsilon\geq\displaystyle \limsup_{m\rightarrow\infty}F\bigg{(}\frac{u_{n_{m}}+u}{2}\bigg{)}.
	\end{equation}
	and that is a contradiction with \eqref{11}. It follows that  $(u_{n})$ converges strongly to $u$ in $E$ and lemma \ref{lem3} is proved.
\end{proof}

\section{Application to non-local fractional problems }

The main task of this Section is to prove Theorem \ref{thm}.

We shall work in the closed linear subspace
$$\tilde{W}_{0}^{s,M}(\Omega)=\{u\in W^{s,M}_{0}(\Omega):\ u=0\ a.e\ \text{in}\ \mathbb{R}^{N}\setminus\Omega\}$$
equivalently renormed by setting $\|.\|:=[.]_{s,M}$, which is a reflexive separable Banach space. \\

\begin{remark}\label{rem}
Invoking condition $(Q)$ and Theorem \ref{thm1}, we deduce that $\tilde{W}^{s,M}_{0}(\Omega)$ is
compactly embedded in $L^{q}(\Omega)$.
\end{remark}

This makes the following definition well-defined.
\begin{definition}
	We say that $u$ is a weak solution to \eqref{eq} if $u\in \tilde{W}_{0}^{s,M}(\Omega)$ and
	$$\langle F^{'}(u),v\rangle-\lambda\int_{\Omega}g(x,u)vdx=0,$$
	for every $v\in \tilde{W}_{0}^{s,M}(\Omega)$.
\end{definition}

For each $\lambda>0$ we define the energy functional $I_{\lambda}: \tilde{W}_{0}^{s,M}(\Omega)\rightarrow \mathbb{R}$ associated to \eqref{eq} given by
$$I_{\lambda}(u)= F(u)-\lambda\int_{\Omega}G(x,u)dx.$$

We first establish some basic properties of $I_{\lambda}$.
\begin{prp}
	For each $\lambda > 0$ the functional $I_{\lambda}>0$ is well-defined on $\tilde{W}_{0}^{s,M}(\Omega)$ and $I_{\lambda}\in C^{1}(\tilde{W}_{0}^{s,M}(\Omega),\mathbb{R})$ with the derivative given by
	$$\langle I_{\lambda}^{'}(u),v\rangle=\langle F^{'}(u),v\rangle-\lambda\int_{\Omega}g(x,u)vdx,$$ for all $u,v\in\tilde{W}_{0}^{s,M}(\Omega) $.
\end{prp}
\begin{proof}
	The proof follows from Lemma \ref{lem1} and condition $(A)$.
\end{proof}
\begin{prp}\label{coe}
	The functional $I_{\lambda}$ is coercive.
\end{prp}
\begin{proof}
	Let $u\in \tilde{W}_{0}^{s,M}(\Omega)$ with $\|u\|>1$. By combining $(ii)$ in Lemma \ref{lem2} and hypothesis $(B)$, we get
	\begin{align*}
	I_{\lambda}(u) &= F(u)-\lambda\int_{\Omega}G(x,u)dx\\
	&\geq \|u\|^{m_{0}}-\lambda C_{2}\|u\|_{L^{q}(\Omega)}^{q}.
	\end{align*}
	Since $q<m_{0}$ the above inequality implies that $I_{\lambda}(u)\rightarrow \infty$ as $\|u\|\rightarrow \infty$, that is, $I_{\lambda}$ is coercive.
\end{proof}
\begin{prp}\label{sci}
	The functional $I_{\lambda}$ is weakly lower semi-continuous.
\end{prp}
\begin{proof}
	Let $u_{n}\subset\tilde{W}_{0}^{s,M}(\Omega)$ be a sequence which converges weakly to $u$ in $\tilde{W}_{0}^{s,M}(\Omega)$. By Lemma \ref{lem4}, we deduce that \begin{equation}\label{s1}
	F(u)\leq \liminf_{n\rightarrow+\infty}F(u_{n}).
	\end{equation}
	On the other hand, Remark \ref{rem} and conditions $(A)$ and $(B)$ imply
	\begin{equation}\label{s2}
	\lim_{n\rightarrow+\infty}\int_{\Omega}G(x,u_{n})dx=\int_{\Omega}G(x,u)dx.
	\end{equation}
	
	Thus, from \eqref{s1} and \eqref{s2}, we find $$I_{\lambda}(u)\leq \liminf_{n\rightarrow+\infty}I_{\lambda}(u_{n}).$$
	Therefore, $I_{\lambda}$ is weakly lower semi-continuous and Proposition \ref{sci} is verified.
\end{proof}
From Proposition \ref{coe} and \ref{sci} and Theorem 1.2 in \cite{28} we deduce that there exists $u_{1}\in \tilde{W}_{0}^{s,M}(\Omega)$ a global minimizer of $I_{\lambda}$. The following result implies that $u_{1}\neq0$.
\begin{prp}\label{energie}
	For every $\lambda>0$ we have $\inf_{\tilde{W}_{0}^{s,M}(\Omega)}I_{\lambda}<0$.
\end{prp}
\begin{proof}
	Fix $v\in\tilde{W}_{0}^{s,M}(\Omega)$, $v\neq0$ and $v\geq 0$ in $\Omega$. Using relation $(iii)$ in Lemma \ref{lem2} and condition $(B)$ we obtain \begin{align*}
	I_{\lambda}(tv) &= F(tv)-\lambda\int_{\Omega}G(x,tv)dx\\
	&\leq t^{m_{0}}\|v\|^{m_{0}}-\lambda C_{2} t^{q}\|v\|^{q},
	\end{align*}
	for $t$ small enough. Taking into account $q<m_{0}$, we infer that $I_{\lambda}(tv)<0.$ The proof of Proposition \ref{energie} is complete.
\end{proof}
Since Proposition \ref{energie} holds it follows that $u_{1}\in\tilde{W}_{0}^{s,M}(\Omega)$ is a non-trivial weak solution of problem \eqref{eq}.

\begin{lemma}\label{geo1}
	Assume the hypotheses of Theorem \ref{thm} are fulfilled. Then there exists $\lambda^{*} > 0$ such that for any $\lambda\in ]0,\lambda_{*}[$ there exist $\rho,\alpha>0$ such that $I_{\lambda}(u)\geq \alpha>0$ for any $u\in\tilde{W}_{0}^{s,M}(\Omega) $ with $\|u\|=\rho$.
\end{lemma}
\begin{proof}
	In light of Remark \ref{rem}, there exists a positive constant $c_{1}$ such that $$\|u\|_{L^{q}(\Omega)}\leq c_{1}\|u\|,\ \forall u\in \tilde{W}_{0}^{s,M}(\Omega).$$
	
	We fix $\rho\in]0,\|u_{1}\|[$.\\
	\textbf{Case $1$: } $\|u_{1}\|<1$.
	Invoking $(iii)$ in Lemma \ref{lem2} and $(A)$, we deduce that
	\begin{align*}
	I_{\lambda}(u)&\geq \|u\|^{m^{0}}-\lambda C_{2} c_{1}^{q}\|u\|^{q}\\&=\rho^{q}(\rho^{m^{0}-q}-\lambda C_{2}c_{1}^{q}),
	\end{align*}
	
	for any $u\in\tilde{W}_{0}^{s,M}(\Omega) $ with $\|u\|=\rho$. Put $\lambda^{*}=\frac{\rho^{m^{0}-q}}{3 C_{2}c_{1}^{q}}$.
	Then, for any $\lambda\in ]0,\lambda^{*}[$, we obtain
	$$I_{\lambda}(u)\geq \alpha>0, \ \ \forall u\in \tilde{W}_{0}^{s,M}(\Omega) \ \ \mbox{and} \ \ \|u\|=\rho,$$
	where $\alpha=\displaystyle\frac{\rho^{m^{0}-q}}{3}$.\\
	\textbf{ Case $2$:} $1<\|u_{1}\|$. It sufficient to replace $m^{0}$ by $m_{0}$ in the previous case.
	
	This ends the proof.
\end{proof}

\begin{proof}[\bf Proof of Theorem \ref{thm} completed.]
	Using Lemma \ref{geo1} and the Mountain Pass Theorem (see Theorem 2.1 in \cite{9}) we deduce that there exists a sequence $(u_{n})\subset\tilde{W}_{0}^{s,M}(\Omega)$ such that \begin{equation}\label{ps}
	I_{\lambda}(u_{n})\rightarrow c>0\ \ \text{and}\ \  I_{\lambda}^{'}(u_{n})\rightarrow0,
	\end{equation}
	where $$ c = \inf_{\gamma \in \Gamma} \max_{0 \leq t \leq 1} J( \gamma(t))$$
	and  $$ \Gamma= \left\{ \gamma \in \mathcal{C}([0,1],X),\ \gamma(0) = 0,\ \gamma(1) =u_{1}\right\}. $$
	By relation \eqref{ps} and proposition \ref{coe} we obtain that $(u_{n})$ is bounded and thus passing eventually to a subsequence, still denoted by  $(u_{n})$, we may assume that there exists $u_{2}\in \tilde{W}_{0}^{s,M}(\Omega)$ such that $u_{n}$ converges weakly to $u_{2}$. 
	Hence
	\begin{align*}
	\langle I_{\lambda}^{'}(u_{n})- I_{\lambda}^{'}(u_{2}),u_{n}-u_{2}\rangle&=\langle F^{'}(u_{n})-F^{'}(u_{2}),u_{n}-u_{2}\rangle\\&-\lambda\int_{\Omega}[g(x,u_{n})-g(x,u_{2})](u_{n}-u_{2})dx\rightarrow0,\ n\rightarrow+\infty,
	\end{align*}
	
	where $F$ is defined in relation \eqref{F}. Therefore, by combining Remark \ref{rem} and Lemma \ref{lem3},  we can deduce that $u_{n}$ converges strongly to $u_{2}$ in $\tilde{W}_{0}^{s,M}(\Omega)$. It follows, in view of relation \eqref{ps}, that
	$$I_{\lambda}(u_{2})=c>0\ \ \text{and}\ \ I_{\lambda}^{'}(u_{2})=0.$$
	We conclude that $u_{2}$ is a critical point of $I_{\lambda}$ and so it is a non trivial second solution of \eqref{eq}. Since $I_{\lambda}(u_{1})<0$, we can conclude that $u_{2}\neq u_{1}$. The proof of Theorem \ref{thm} is now complete.
\end{proof}

\bigskip
\noindent \textsc{\textsc{Sabri Bahrouni and Hichem Ounaies}}\\
Mathematics Department, Faculty of Sciences, University of Monastir,
5019 Monastir, Tunisia\\
 (sabri.bahrouni@fsm.rnu.tn, hichem.ounaies@fsm.rnu.tn)

\noindent \mbox{} \\
\noindent and \\
\noindent \mbox{} \\
\noindent \textsc{Leandro  S. Tavares}\\
Centro de Ci\^{e}ncias e Tecnologia\\
Universidade Federal do Cariri\\
Juazeiro do Norte, CE, Brazil,\\
CEP:63048-080 \\
\texttt{leandro.tavares@ufca.edu.br} \\
\noindent \mbox{} \\

\end{document}